\documentclass[12pt]{amsart}

\usepackage[T1]{fontenc}
\usepackage{lmodern}
\usepackage[a4paper,margin=1.6cm]{geometry}
\usepackage{amsmath,amssymb,mathtools,amsthm}
\usepackage{microtype}
\usepackage[hidelinks]{hyperref}
\hypersetup{pdftitle={The Fractional-Hessian Sobolev Inequality},
	pdfsubject={Measurable coefficients, Sobolev Hessians, and the measure endpoint}}

\newcommand{\tr}{\operatorname{tr}}
\newcommand{\diam}{\operatorname{diam}}
\newcommand{\Sym}{\operatorname{Sym}}
\newcommand{\sym}{\operatorname{sym}}
\newcommand{\gap}{\operatorname{gap}}
\newcommand{\inc}{\operatorname{inc}}
\newcommand{\R}{\mathbb R}
\newcommand{\D}{\mathbb D}

\newcommand{\op}{\mathrm{op}}
\newcommand{\per}{\mathrm{per}}
\newcommand{\loc}{\mathrm{loc}}
\newcommand{\dx}{\,dx}
\DeclareMathOperator{\cof}{cof}
\newtheorem{theorem}{Theorem}
\newtheorem{lemma}[theorem]{Lemma}
\newtheorem{proposition}[theorem]{Proposition}
\theoremstyle{remark}
\newtheorem*{remark}{Remark}
\numberwithin{equation}{section}
\allowdisplaybreaks[1]
\title[Fractional elementary symmetric functions and $q$-Hessian estimates]{Fractional elementary symmetric functions and \\ Hessian estimates in the plane}
\author{Gautam Neelakantan Memana}
\address{Mathematical Sciences Institute, Australian National University, Canberra, ACT 2601}
\email{neelanmemana@gmail.com}
\author{Neil S. Trudinger}
\address{Mathematical Sciences Institute, Australian National University, Canberra, ACT 2601}
\email{Neil.Trudinger@anu.edu.au}
\author{Po-Lam Yung}
\address{Mathematical Sciences Institute, Australian National University, Canberra, ACT 2601 \& CNRS-ANU International Research Laboratory FAMSI (France Australia Mathematical Sciences and Interactions) \& Department of Mathematics, The Chinese University of Hong Kong}
\email{PoLam.Yung@anu.edu.au, plyung@math.cuhk.edu.hk}
\date{12 October 2026}

\begin{document}
	\begin{abstract}
		We prove an endpoint Pucci maximum principle for domains in $\R^2$,
		formulated as in \cite{Trudinger}, but using
		a generalization of elementary symmetric polynomials to noninteger orders.  
		As a result, we sharpen an earlier result of Astala, Iwaniec and Martin \cite{AIM}, and 
		deduce a new Sobolev inequality for $q$-Hessians where $1 < q < 2$.
		The fractional elementary symmetric functions we use are derived through fractional differentiation. Their signed singular value extension coincides with the Burkholder functional introduced by Burkholder \cite{Burkholder1984} and gives us access to a local Burkholder inequality due to Astala, Iwaniec, Prause and Saksman \cite{AIPS}, leading to a different proof of the Pucci maximum principle from the one in \cite{AIM} that used quasiconformal maps.
		We include an elementary real-variable proof of this local Burkholder inequality, obtained by combining a few known techniques from \cite{AA,Burkholder,Sverak,BMS,Guerra}. 
	\end{abstract}
	\maketitle
	
	\section{Introduction}
	
	In this paper, we introduce a generalization of the Laplacian and the Monge--Amp\`ere operator in $\R^2$, and use it to prove an endpoint Pucci maximum principle, sharpening an earlier result of Astala--Iwaniec--Martin \cite{AIM}.
	
	The Laplacian and the Monge--Amp\`ere operator of a smooth function $u$ in $\R^2$ are defined by taking the trace and the determinant of the Hessian matrix $D^2 u$. The trace and the determinant are elementary symmetric polynomials of the eigenvalues of $D^2 u$, and our first contribution is in identifying a generalization of elementary symmetric polynomials to elementary symmetric functions of non-integer orders. 
	
	Let $\Sym(2)$ denote the real symmetric $2\times2$ matrices, with inner
	product $\tr(AH)=\sum_{i,j}A_{ij}H_{ij}$. For $q>0$ and
	$H\in\Sym(2)$ with $\lambda_{\max}(H)>0$, we define
	\begin{equation}\label{eq:density}
		S_q(H)=\lambda_{\max}(H)^{q-1}
		\left(\frac q2\lambda_{\min}(H)
		+\left(1-\frac q2\right)\lambda_{\max}(H)\right),
	\end{equation}
	and set $S_q(0)=0$; henceforth $\lambda_{\max}$ and $\lambda_{\min}$ represent the largest and smallest eigenvalue of a real symmetric matrix. For $p\geq2$ and $P\in\R^{2\times2}$, we use the
	same notation for its Burkholder extension
	\[
	S_p(P)=M^{p-1}\left(\frac p2m+\left(1-\frac p2\right)M\right),
	\]
	where $M\geq|m|$ are the signed singular values of $P$, determined by
	$M=\|P\|_{\op}$ and $Mm=\det P$. The two definitions agree whenever
	$P$ is symmetric and
	$\lambda_{\max}(P)\geq|\lambda_{\min}(P)|$, in particular for
	$P\in\Gamma_p$ for $p \geq 1$, where 
	\[
	\Gamma_q :=\left\{H\in\Sym(2):
	\frac q2\lambda_{\min}(H)+\left(1-\frac q2\right)
	\lambda_{\max}(H)\geq0\right\} \qquad \text{for $q \geq 0$}
	\] 
	is our generalization of the Garding cone to non-integer orders.
	We have $S_1(H)=\tr(H)/2$, $S_2(P)=\det P$, and
	$S_q(I)=1$. We set $S_0(H):=1$ for $H\in\Sym(2)$.
	$S_q(D^2 u)$ is then our fractional degree generalization of the Laplacian and the Monge--Amp\`ere operator: for $1 < q \leq 2$, $S_q^{1/q}(D^2 u)$ is elliptic and concave when $D^2 u \in \Gamma_q$.
	See Section~\ref{sec:elliptic} below.
	
	The functional $S_p$, $p \geq 1$, has already appeared in the work of Burkholder \cite{Burkholder1984} as an auxiliary device; see \cite[p.661, (5.33)]{Burkholder1984} and also \cite{AFGKK}. We arrived at $S_q$, $q > 0$, via a somewhat different route, via fractional differentiation; see Section~\ref{sec:fractionaldiff}. The fractional differentiation point of view turned out to connect with another useful integral representation of $S_p$ in Baernstein–Montgomery-Smith \cite[Section 1, equation (1.2b)]{BMS}; see also Guerra \cite[proof of Corollary 1.3]{Guerra}. This connection plays a key role in a real-variable proof of a local Burkholder principle; see Section~\ref{sec:burkholder}. Our endpoint Pucci maximum principle ultimately relies on this local Burkholder principle.

	To formulate our endpoint Pucci maximum principle, we need instead the dual cone to $\Gamma_q$, given for $q>1$ by
	\[
	\Gamma_q^*=\left\{A\in\Sym(2):
	\frac q2\lambda_{\min}(A)-\left(1-\frac q2\right)\lambda_{\max}(A)\geq0\right\},
	\]
	and the dual functional to $S_q$, given in the same range of $q$ by
	\[
	\rho_q^*(A)=\frac{\lambda_{\max}(A)^{1/q}
		\left(\frac q2\lambda_{\min}(A)-\left(1-\frac q2\right)\lambda_{\max}(A)\right)^{1-1/q}}
	{(q-1)^{1-1/q}}, \qquad A \in \Gamma_q^*.
	\]
	
	The Pucci maximum principle established by Astala--Iwaniec--Martin \cite{AIM} can be stated as follows. Let $1<q_0<q\leq2$ and let
	$\Omega\subset\R^2$ be a nonempty bounded open set. If $A$ is measurable,
	$A(x)\in\Gamma_{q_0}^*\setminus\{0\}$ a.e., and
	$u\in W^{2,q}(\Omega)\cap C(\overline\Omega)$ satisfies
	$\tr(AD^2u)=f$, then \cite[Theorem~3.5]{AIM} implies that
	\begin{equation}\label{eq:AIM-comparison}
		\|u\|_{L^\infty(\Omega)}
		\leq \max_{\partial\Omega}|u|
		+C_{q,q_0}\,\diam(\Omega)^{2-2/q}
		\left\|\frac{f}{\tr (A)}\right\|_{L^q(\Omega)}.
	\end{equation}
	For $A \in \Gamma_{q_0}^*$ and $1 < q_0 < q \leq 2$, one can check that
	\[
	\rho_q^*(A)\geq c_{q,q_0}\operatorname{tr}(A),\qquad c_{q,q_0} > 0;
	\]
	see \eqref{eq:rho-trace-comparison}.
	In light of this, our result below can be seen as the endpoint case when $q = q_0$.
	
	\begin{theorem}[Endpoint Pucci maximum principle]\label{thm:weighted-maximum}
		Let $1<q\leq2$ and let $\Omega\subset\R^2$ be a nonempty bounded open
		set. Suppose $A:\Omega\to\Sym(2)$ is measurable and
		$A(x)\in\operatorname{int}\Gamma_q^*$ a.e. If
		$u\in W^{2,q}_{\loc}(\Omega)\cap C(\overline\Omega)$ satisfies
		$\tr(AD^2u)=f \text{  a.e. in }\Omega$,
		then
		\begin{equation}\label{eq:weighted-maximum}
			\sup_\Omega u
			\leq \sup_{\partial\Omega}u
			+C_q\,\diam(\Omega)^{2-2/q}
			\left\|\frac{f^-}{\rho_q^*(A)}\right\|_{L^q(\Omega)},
			\qquad f^-:=\max\{-f,0\}.
		\end{equation}
	\end{theorem}
	
	At $q=2$, $\Gamma_2=\Gamma_2^*=\{H\geq0\}$ and
	$\rho_2^*(A)=\sqrt{\det A}$, so Theorem~\ref{thm:weighted-maximum} reduces to the Alexandrov estimate
	\[
	\sup_\Omega u
	\leq \sup_{\partial\Omega}u
	+C \diam(\Omega) \left\|\frac{f^-}{\sqrt{\det A}}\right\|_{L^2(\Omega)}.
	\]
	See also \cite{Trudinger,Trudinger2024} for prior work on Pucci's
	estimates in higher dimensions. Also while our work was in progress, the full resolution of Pucci's conjecture in higher dimensions was  accomplished in two substantially different and independent approaches by Chen, Trudinger and Wang \cite{CTW 2026} and Hung and Wu \cite{Hung-Wu}.
	
	Applying Theorem~\ref{thm:weighted-maximum} with $A_\varepsilon=D(S_q^{1/q})(D^2u+\varepsilon I)$ and using the ellipticity and concavity of $S_q^{1/q}$ gives the following $q$-Hessian estimate. 
	
	\begin{theorem}[A Sobolev estimate for $S_q$]\label{thm:Sobolev}
		Let $1<q\leq2$ and let $\Omega\subset\R^2$ be a nonempty bounded open
		set. Suppose $u\in W^{2,q}_{\loc}(\Omega)\cap C(\overline\Omega)$,
		$u=0$ on $\partial\Omega$, $D^2u\in\Gamma_q$ a.e., and
		$S_q(D^2u)\in L^1(\Omega)$. Then
		\begin{equation}\label{eq:target}
			\|u\|_{L^\infty(\Omega)}
			\leq C_q\,\diam(\Omega)^{2-2/q}
			\left(\int_\Omega S_q(D^2u)\dx\right)^{1/q}.
		\end{equation}
	\end{theorem}
	
	At $q=2$, there is also a measure formulation: if $u\in C(\overline\Omega)$ is locally convex in $\Omega$, vanishes on $\partial\Omega$, and has finite Monge--Amp\`ere mass, then
	\begin{equation}\label{eq:MA-target}
		\|u\|_{L^\infty(\Omega)} \leq C\,\diam(\Omega)\,\mu_2[u](\Omega)^{1/2}.
	\end{equation}
	This is the case $n=k=2$ of \cite{TW2}.
	Here locally convex means convex on every ball contained in $\Omega$,
	and $\mu_2[u]$ is the Monge--Amp\`ere measure. It equals
	$\det D^2u\,dx$ for smooth $u$ and is weakly continuous under locally
	uniform convergence of locally convex functions
	\cite[Theorem~1.1]{TW2}; see also \cite[Section~8]{WangHessian}.
	Requiring only $u\in C(\overline\Omega)$ and local convexity is much
	weaker than the hypothesis $u\in W^{2,2}_{\loc}(\Omega)$ in
	Theorem~\ref{thm:Sobolev}: a locally convex function may have singular
	Monge--Amp\`ere mass and need not belong to $W^{2,2}_{\loc}(\Omega)$.
	A similar issue arises when $1 < q < 2$, as exemplified by the following example: if $u(x)= |x|^{2-2/q}-1$, then it is continuous throughout $\D$, vanishes on the boundary, $D^2 u(x) \in \Gamma_q$ with $S_q(D^2 u)(x) = 0$ for all $x \ne 0$. Nonetheless of course $\|u\|_{L^{\infty}(\D)} \ne 0$. This does not contradict Theorem~\ref{thm:Sobolev} because this $u$ fails to be $W^{2,q}$ near $0$. When $q = 2$, this function $u(x)=|x|-1$ is captured by the measure formulation of \cite{TW2}: $S_q(D^2 u)$ is then the measure $\pi \delta_0$. It is thus tempting to try to formulate a measure version of our Theorem~\ref{thm:Sobolev} as well when $1 < q < 2$. Unfortunately this measure construction does not
	extend to $S_q$ for $1<q<2$; see Section~\ref{sec:measures} for a simple example. 
	
	One distinctive feature for our $S_q(D^2 u)$ when $q$ is not an integer is that it lacks the divergence structure enjoyed by the elementary symmetric polynomials $S_k$ when $k$ is a positive integer. To get around this, our proof of Theorem~\ref{thm:weighted-maximum} relies on the following local Burkholder principle on the unit disk $\D$.
	\begin{theorem}[Local Burkholder principle on the disk, \cite{AFGKK}]\label{thm:local-burkholder}
		Let $2\leq p<\infty$ and let $A\in\R^{2\times2}$. If
		$f-Ax\in W^{1,2}_0(\D;\R^2)$ and $S_p(Df)\geq0$ a.e. in $\D$, then
		$S_p(Df)\in L^1(\D)$ and
		\[\frac1{|\D|}\int_\D S_p(Df)\dx\leq S_p(A).\ \]
		For $p=2$, equality holds.
	\end{theorem}
	We also note that by rescaling, Theorem~\ref{thm:local-burkholder}
	holds with $\D$ replaced by $r\D$ for every $r>0$.
	
	The case $A=I$ (the identity matrix) was proved by Astala--Iwaniec--Prause--Saksman \cite{AIPS} using the Beltrami equation and complex interpolation. The case of general $A$ is due to Astala--Faraco--Guerra--Koski--Kristensen \cite{AFGKK}, using principal quasiconformal maps and Young measures. We will only need the case $A=I$, but we state the theorem for general $A$ in order to extract from \cite{AA} a real-variable proof of the full result. The argument in \cite{AA} establishes the corresponding inequality when the perturbation $f-Ax$ belongs to \(C^\infty(\overline\D;\R^2)\) and vanishes on \(\partial \D\). This direct, real-variable proof, presented in Section~\ref{sec:burkholder} below, may have a better chance of being extended to higher dimensions. In fact, the assumption $S_p(Df) \geq 0$ is not even needed in \cite[Corollary 2]{AA} when $f$ is smooth.

	
	%
	%
	
	From the case $A=I$ of Theorem~\ref{thm:local-burkholder}, we deduce
	the following estimate, which is dual to the Pucci maximum principle
	in Theorem~\ref{thm:weighted-maximum}. For a symmetric matrix field
	$B$, define
	\[
	\operatorname{div}\operatorname{div}B
	=\partial_{11}B_{11}+2\partial_{12}B_{12}+\partial_{22}B_{22}.
	\]
	Here and below these derivatives are distributional when the field is
	only in $L^2$. Write $\D=B_1(0)$.
	
	\begin{proposition}[Dual form of the Pucci maximum principle]\label{prop:dual-burkholder}
		Let $1<q\leq2$ and $q'=q/(q-1)$. Suppose
		\[
		B\in L^2(\D;\Sym(2))\cap C^\infty(\overline\D\setminus\{0\};\Sym(2)),
		\qquad B=0\quad\text{on }\partial\D,
		\]
		where the smoothness assertion includes smoothness up to the outer
		boundary, away from the pole. Assume
		\[
		\operatorname{div}\operatorname{div}B=-\delta_0,
		\qquad B(x)\in\Gamma_q^*\quad\text{a.e. in }\D.
		\]
		Then
		\begin{equation}\label{eq:dual-burkholder}
			\int_\D\rho_q^*(B)^{q'}\dx\leq C_q.
		\end{equation}
	\end{proposition}
	
	The article is organized as follows. Section~\ref{sec:Sq} contains our derivation of $S_q$ and information about its dual $\rho_q^*$, along with a few other interesting properties, such as an $S_q$ analog of Newton's and Maclaurin's inequalities for elementary symmetric functions, as well as the ellipticity and concavity of the $q$-Hessian equation $S_q^{1/q}(D^2 u) = f$. Section~\ref{sec:strain} deduces Proposition~\ref{prop:dual-burkholder} from the $A=I$ case of Theorem~\ref{thm:local-burkholder}, and Section~\ref{sec:pf_thm1} proves Theorem~\ref{thm:weighted-maximum} from Proposition~\ref{prop:dual-burkholder} using duality. Section~\ref{sec:hessian-proof} proves Theorem~\ref{thm:Sobolev} from Theorem~\ref{thm:weighted-maximum} by linearization. Section~\ref{sec:burkholder} provides an elementary, real-variable proof of Theorem~\ref{thm:local-burkholder}, which is self-contained apart from the one page proof of Lemma~\ref{lem:projection} that is already very nicely explained in \cite{AA}. Finally in Section~\ref{sec:measures} we explain why the Hessian measure construction of Wang and the second author \cite{TW2} does not extend to $S_q$ for $1<q<2$.
	
	%
	%
	%

	\medskip
	
	\noindent{\bf Acknowledgments.} The authors were partially supported by a Discovery Project DP250103744 from the Australian Research Council.
	
	\medskip
	
	\noindent{\bf Statement on AI use.}
	The second author started this project by suggesting that the third author look for an operator whose fundamental solution is $|x|^{2-2/q}$, around 2023.
	The third author then found $S_q$ defined as above using fractional differentiation, and used Mathematica to verify its various properties. No AI was used up to this point.
	In August--September 2026, the first author identified $S_q$ as
	the Burkholder functional in \cite{Burkholder1984, AFGKK}, and after suggesting a linearization strategy, obtained a ChatGPT-5.6 Sol assisted
	proof of Theorem~\ref{thm:Sobolev} using the local Burkholder inequality (Theorem~\ref{thm:local-burkholder})
	as an input.
	The proof was then further understood with more help from ChatGPT-6.1 Sol, leading to the current formulation of Theorem~\ref{thm:weighted-maximum} and the dual
	Proposition~\ref{prop:dual-burkholder}. Subsequent literature searches
	identified the ingredients for the real-variable proof of
	Theorem~\ref{thm:local-burkholder} presented in
	Section~\ref{sec:burkholder}; the proof is primarily the one found in \cite{AA}, with inputs from several earlier works including
	\cite{AIPS,Burkholder,Sverak,BMS,Guerra}. 
	
	\section{Fractional elementary symmetric functions} \label{sec:Sq}
	
	For $q \ge 0$, define
	\begin{align*}
		\Gamma_q&=\left\{H\in\Sym(2):
		\frac q2\lambda_{\min}(H)+\left(1-\frac q2\right)
		\lambda_{\max}(H)\geq0\right\}
	\end{align*}
	so that $S_q(H) \geq 0$ for all $H \in \Gamma_q$, with strict inequality for $H \in \operatorname{int} \Gamma_q$. Note that $\Gamma_q$ is a convex cone if $q \geq 1$. In that case define the dual cone
	\begin{align*}
		\Gamma_q^*&=\{A\in\Sym(2):\tr(AH)\geq0\text{ for every }H\in\Gamma_q\}.
	\end{align*}
	For $A\in\Gamma_q^*$, set
	\begin{equation}\label{eq:rho-definition}
		\rho_q^*(A)=\inf \Big\{ 
		\frac{\tr(AH)}{2S_q(H)^{1/q}} \colon H\in\operatorname{int}\Gamma_q \Big\}.
	\end{equation}
	For $q>1$, let $q' := q/(q-1)$. Then
	\begin{equation} \label{eq:Gamma_q^*Gamma_{q'}}
		\Gamma_q^*=\left\{A\in\Sym(2):
		\frac q2\lambda_{\min}(A)-\left(1-\frac q2\right)\lambda_{\max}(A)\geq0\right\} = \Gamma_{q'},
	\end{equation}
	and for $A \in \Gamma_q^*$ we have
	\begin{equation}\label{eq:rho-eigenvalue-formula}
		\rho_q^*(A)=\frac{\lambda_{\max}(A)^{1/q}
			\left(\frac q2\lambda_{\min}(A)-\left(1-\frac q2\right)\lambda_{\max}(A)\right)^{1/q^{\prime}}}
		{(q-1)^{1/q^{\prime}}} = S_{q'}(A)^{1/q'}.
	\end{equation}
	For $1<q_0<q\leq2$ and $A \in \Gamma_{q_0}^*$, the same formula gives the uniform lower bound
	\begin{equation}\label{eq:rho-trace-comparison}
		\rho_q^*(A)\geq c_{q,q_0}\operatorname{tr}(A),\qquad c_{q,q_0} > 0.
	\end{equation}
	We record in the next subsection the derivations of
	\eqref{eq:Gamma_q^*Gamma_{q'}}, \eqref{eq:rho-eigenvalue-formula}, and
	\eqref{eq:rho-trace-comparison}, along with a few additional useful properties for $\rho_q^*$. We emphasize though that the identity of $\Gamma_q^*$ with $\Gamma_{q'}$, and that of $\rho_q^*$ with a power of $S_{q'}$, is very much a 2 dimensional phenomena, that should not be expected in higher dimensions.

	\subsection{The dual functional \texorpdfstring{$\rho_q^*$}{rho-q-star}}
	
	Let $q>1$, and let $A\in\Sym(2)$ have eigenvalues $a\geq b$.
	If $H\in\Gamma_q$ has eigenvalues $x\geq y$, then $x\geq0$ and
	$y/x\in[1-2/q,1]$ when $x>0$. If $\theta$ is the angle between the
	$a$-eigenvector of $A$ and the $x$-eigenvector of $H$, then
	\[
	\operatorname{tr}(AH)=ay+bx+(a-b)(x-y)\cos^2\theta\geq ay+bx.
	\]
	Consequently, $A\in\Gamma_q^*$ if and only if
	\[
	ar+b\geq0\qquad\left(1-\frac2q\leq r\leq1\right).
	\]
	Indeed, its value at $r=1$ is $a+b\geq0$, which together with
	$b\leq a$ forces $a\geq0$. Hence the left-hand side is
	nondecreasing in $r$ and it is enough to test the left endpoint. We
	obtain
	\[
	A\in\Gamma_q^*\quad\Longleftrightarrow\quad
	\frac q2b-\left(1-\frac q2\right)a\geq0.
	\]
	The identity
	\begin{equation}\label{eq:dual-bracket-identity}
		\frac q2b-\left(1-\frac q2\right)a
		=(q-1)\left[
		\frac{q'}2b+\left(1-\frac{q'}2\right)a\right]
	\end{equation}
	also proves $\Gamma_q^*=\Gamma_{q'}$, completing the proof of
	\eqref{eq:Gamma_q^*Gamma_{q'}}.
	
	Now let $A\in\Gamma_q^*$. The preceding orientation calculation shows
	that the infimum in \eqref{eq:rho-definition} over the relative
	eigenvector orientation is attained when the $y$-eigenvector of $H$
	is aligned with the $a$-eigenvector of $A$ (equivalently, the $x$-
	and $b$-eigenvectors are aligned), and its value is $ay+bx$; in other
	words,
	\[
	\rho_q^*(A) = \inf_{\substack{\operatorname{diag} (x,y) \in \operatorname{int} \Gamma_q \\ x \geq y}} \frac{ay+bx}{2x^{(q-1)/q} \left(\frac q2 y + (1-\frac q2) x\right)^{1/q}}.
	\]
	
	Since the quotient is homogeneous in $(x,y)$, set $r=y/x$. The conditions $\operatorname{diag} (x,y) \in\operatorname{int}\Gamma_q$ and $x \geq y$ are $1-\frac{2}{q}<r\leq1$.
	Thus
	\begin{equation} \label{eq:rhoqAvar}
		\rho_q^*(A)=\inf_{1-2/q<r\leq1}\frac{b+ar}{2\left(\frac q2r+1-\frac q2\right)^{1/q}}.
	\end{equation}
	The formula is immediate when $A=0$, so assume $A\ne0$. Then $a>0$.
	Differentiating the above quotient gives
	\[
	\frac{d}{dr}\left[\frac{b+ar}{2\left(\frac q2r+1-\frac q2\right)^{1/q}}\right]=\frac{a(q-1)(r-1)+a-b}{4(\frac q2r+1-\frac q2)^{1/q+1}}.
	\]
	The denominator is positive, while the numerator is strictly increasing in $r$. Hence the quotient decreases before its unique critical point and increases afterward. The inequalities $b\leq a$ and
	$\frac q2b-(1-\frac q2)a\geq0$ give
	\[
	0\leq\frac{a-b}{(q-1)a}\leq\frac2q.
	\]
	The critical point is therefore
	\[
	r_*=1-\frac{a-b}{(q-1)a}\in[1-2/q,1].
	\]
	By continuity in $r$, the infimum in \eqref{eq:rhoqAvar} is given by substituting $r = r_*$. This gives
	\begin{equation}\label{eq:rho-gap-formula}
		\rho_q^*(A)=\frac{a^{1/q}
			\left(\frac q2b-\left(1-\frac q2\right)a\right)^{1/q^{\prime}}}
		{(q-1)^{1/q^{\prime}}}
	\end{equation}
	which is the first equality in \eqref{eq:rho-eigenvalue-formula}. The
	second follows from the definition of $S_{q'}$, equation
	\eqref{eq:dual-bracket-identity}, and
	$1/q=(q'-1)/q'$.

	For $1<q_0<q\leq2$ and $A\in\Gamma_{q_0}^*$, we have
	\[
	\frac q2b-\left(1-\frac q2\right)a
	=\frac {q_0}2b-\left(1-\frac {q_0}2\right)a
	+\frac{q-q_0}{2}(a+b)
	\geq\frac{q-q_0}{2}(a+b).
	\]
	Substituting this and $a\geq(a+b)/2$ in
	\eqref{eq:rho-gap-formula} gives \eqref{eq:rho-trace-comparison}
	for some $c_{q,q_0}>0$.
	
	The defining infimum \eqref{eq:rho-definition} also makes $\rho_q^*$
	concave. Since $\rho_q^*(I/2)=1/2$, for $A\in\Gamma_q^*$ and
	$0<\varepsilon<1$ we have
	\begin{equation}\label{eq:rho-regularization}
		\rho_q^*((1-\varepsilon)A+\varepsilon I/2)
		\geq (1-\varepsilon)\rho_q^*(A)+\varepsilon/2.
	\end{equation}
	This lower bound will permit approximation of measurable coefficients.

	For later use, we record the corresponding formulas for $S_p$. For
	$p>0$ and $H\in\Sym(2)$, set
	\[
	\gap(H):=\lambda_{\max}(H)-\lambda_{\min}(H).
	\]
	Since
	\[
	\gap(H)=2\left\|H-\frac{\operatorname{tr}(H)}2I\right\|_{\operatorname{op}},
	\]
	the triangle inequality gives
	\begin{equation} \label{eq:gaptriangle}
		\gap(H+K)\leq\gap(H)+\gap(K)
		\qquad \text{for }H,K\in\Sym(2).
	\end{equation}
	Then
	\begin{equation}\label{eq:primal-trace-gap}
		H\in\Gamma_p\quad\Longleftrightarrow\quad
		\operatorname{tr}(H)-(p-1)\gap(H)\geq0.
	\end{equation}
	Moreover, for every $p>0$ and every nonzero $H\in\Gamma_p$,
	\begin{equation}\label{eq:Sp-trace-gap}
		S_p(H)=\frac1{2^p}
		\bigl(\operatorname{tr}(H)+\gap(H)\bigr)^{p-1}
		\bigl(\operatorname{tr}(H)-(p-1)\gap(H)\bigr).
	\end{equation}
	Both identities are algebraic consequences of
	$\operatorname{tr}(H)+\gap(H)=2\lambda_{\max}(H)$. At $H=0$, the
	cone identity still holds and $S_p(0)=0$ by definition.
	
	\subsection{Singular values of a matrix via the trace and eigenvalue gap of its symmetric part}
	
	We shall also use the following comparison with signed singular values. For $P \in \R^{2 \times 2}$ write
	$|P|=(\sum_{i,j}P_{ij}^2)^{1/2}$ for the Frobenius norm. If $M,m$ are
	the signed singular values of a real $2\times2$ matrix $P$, then we will prove
	\begin{align}
		M+m&\geq\tr(P), \label{eq:singular-sum}\\
		M-m&=\gap(\sym P). \label{eq:singular-gap}
	\end{align}
	Here $\sym P$ is the symmetric part of $P$.
	Indeed, write $P=S+\omega J$, where 
	\begin{equation} \label{eq:SymJ}
		S=\sym P :=\frac{P+P^{\mathsf T}}2 \quad \text{ and } \quad J=\begin{pmatrix}0&-1\\1&0\end{pmatrix}.
	\end{equation} 
	The symmetric and skew-symmetric parts are orthogonal for the Frobenius
	inner product, while a direct calculation gives
	\[
	|P|^2=|S|^2+2\omega^2,\qquad
	\det P=\det S+\omega^2.
	\]
	Since $M^2+m^2=|P|^2$ and $Mm=\det P$, it follows that
	\begin{align*}
		(M+m)^2
		&=|P|^2+2\det P
		=|S|^2+2\det S+4\omega^2
		=\tr(P)^2+4\omega^2,\\
		(M-m)^2
		&=|P|^2-2\det P
		=|S|^2-2\det S
		=\gap(S)^2.
	\end{align*}
	Here we used $\tr(P)=\tr(S)$ and the identities
	\[
	\tr(S)^2=|S|^2+2\det S,\qquad
	\gap(S)^2=|S|^2-2\det S.
	\]
	Since $M+m\geq0$ and $M-m\geq0$, this proves
	\eqref{eq:singular-sum}--\eqref{eq:singular-gap}. Thus the symmetric
	part controls the trace and eigenvalue gap combinations appearing in
	\eqref{eq:primal-trace-gap} and \eqref{eq:Sp-trace-gap}, independently
	of the skew part.
	
	\subsection{Fractional differentiation} \label{sec:fractionaldiff}
	
	This project started out with trying to interpolate between the arithmetic and the geometric mean of two positive numbers. Let $\lambda=(\lambda_1, \lambda_2) \in [0,\infty)^2$, and let $$p(t) := (t+\lambda_1) (t+\lambda_2) = S_0(\lambda) t^2 + 2 S_1(\lambda) t + S_2(\lambda)$$ where $S_q$ are the normalized elementary symmetric polynomials
	\[
	S_0(\lambda) = 1, \qquad S_1(\lambda) = \frac{\lambda_1+\lambda_2}{2}, \qquad S_2(\lambda) = \lambda_1 \lambda_2.
	\]
	Observe that
	\[
	S_k(\lambda) := \frac{k!}{2!} \left. \left(\frac{d}{dt} \right)^{2-k} p(t-a) \right|_{t=a}, \quad k = 0,1,2
	\]
	for any $a \in \R$. This suggests using fractional derivatives of
	$p(t-a)$ to define $S_q$ for noninteger $q > 0$. The relevant
	Riemann--Liouville fractional derivative is
	\[ 
	\left(\frac{d}{dt} \right)^{\alpha} p(t) := \frac{1}{\Gamma(2-\alpha)} \left(\frac{d}{dt} \right)^{2} \int_0^t (t-s)^{1-\alpha} p(s) ds, \qquad 0 < \alpha < 2.
	\]
	So taking $\alpha = 2-q$, it motivated us to consider
	\begin{equation}\label{eq:fractional-truncation}
		\begin{split}
			S_{q,a}(\lambda) &:= \frac{\Gamma(q+1)}{2! \Gamma(q)} \left. \left(\frac{d}{dt} \right)^{2} \int_0^t (t-s)^{q-1} p(s-a) ds \right|_{t=a} \\
			&= \frac{1}{2} q (q-1)(q-2) \sum_{k=0}^2
			\binom{2}{k}\frac{(-1)^k}{q-k} S_k(\lambda) a^{q-k}
		\end{split}
	\end{equation}
	for $q > 0$, $a>0$. We sought to define $S_q(\lambda)$ as
	$S_{q,a}(\lambda)$ for a suitable choice of $a$, keeping in mind a
	list of desirable properties (see below). This led us to set
	\[
	S_q(\lambda):=S_{q,\lambda_{\max}}(\lambda),
	\qquad \lambda_{\max}=\max\{\lambda_1,\lambda_2\},
	\]
	for $\lambda\neq(0,0)$. This gives precisely
	\[
	S_q(\lambda) = \lambda_{\text{max}}^{q-1} \left(\frac q2 \lambda_{\text{min}} +\left(1-\frac q2\right) \lambda_{\text{max}} \right),
	\]
	which is how we found the definition \eqref{eq:density} (and the associated $\Gamma_q$ cones). For $p\geq2$, its signed-singular-value extension coincides with the Burkholder functional $H$ in \cite[p.~661, (5.33)]{Burkholder1984}; see also \cite{AFGKK}. Later we will use properties of the Burkholder functional in the proof of Theorem~\ref{thm:local-burkholder} and hence  Theorem~\ref{thm:weighted-maximum}.
	
	The second line of \eqref{eq:fractional-truncation} also makes sense
	for $q>2$ and arbitrary real $\lambda_1,\lambda_2$. It has the
	following integral interpretation. Fix $p>2$. For $a > 0$ and
	$P\in\R^{2\times2}$, set
	\[
	H_{p,a}(P)=\frac12p(p-1)(p-2)
	\int_0^a t^{p-3}(M(P)-t)_+(m(P)-t)\,dt.
	\]
	Here $t_+ := \max\{t,0\}$.
	The weight is integrable at zero because $p>2$. If $0<a\leq M(P)$,
	direct integration gives
	\begin{equation}\label{eq:Hpa-fractional}
		H_{p,a}(P)=\frac12p(p-1)(p-2)
		\sum_{k=0}^2\binom{2}{k}\frac{(-1)^k}{p-k}
		S_k(M(P),m(P))a^{p-k}.
	\end{equation}
	Thus \eqref{eq:Hpa-fractional} says precisely that
	\begin{equation}\label{eq:Hpa-is-Spa}
		H_{p,a}(P)=S_{p,a}(M(P),m(P))\qquad \text{for } 0<a\leq M(P).
	\end{equation}
	
	\begin{lemma}[The integral representation of $S_p(P)$]
		\label{lem:truncated-integral}
		Let $p>2$. For every $P\in\R^{2\times2}$ and $a>0$,
		\begin{equation}\label{eq:truncation-equals-Sp}
			H_{p,a}(P)=S_p(P)\qquad\text{if }a\geq M(P).
		\end{equation}
		Moreover,
		\begin{equation}\label{eq:positive-truncations}
			S_p(P)\geq0\quad\Longrightarrow\quad H_{p,a}(P)\geq0
			\qquad\text{for every }a>0.
		\end{equation}
	\end{lemma}
	\begin{proof}
		If $P=0$, both conclusions are immediate. Otherwise, write
		$M=M(P)>0$ and $m=m(P)$. At $a=M$, equation~\eqref{eq:Hpa-is-Spa}
		and the signed-singular-value definition of $S_p$ give
		$H_{p,M}(P)=S_{p,M}(M,m)=S_p(P)$. For $a\geq M$, the defining integrand
		vanishes on $(M,a)$, which proves
		\eqref{eq:truncation-equals-Sp}.
		
		If $S_p(P)\geq0$, then either $P=0$, or the definition and $p>2$
		imply $0<m\leq M$. As a function of $a$, $H_{p,a}(P)$ increases
		from zero on $(0,m)$, decreases on $(m,M)$ to
		$H_{p,M}(P)=S_p(P)\geq0$, and is constant on $(M,\infty)$. This proves
		\eqref{eq:positive-truncations}.
	\end{proof}

	\subsection{Newton and Maclaurin inequalities for
		\texorpdfstring{$S_q$}{S-q}}
	
	Let $\lambda = (\lambda_1,\lambda_2) \in \R^2$. Write $\lambda_{\max}=\max\{\lambda_1,\lambda_2\}$ and $\lambda_{\min}=\min \{\lambda_1,\lambda_2\}$, and
	assume $\lambda_{\max}>0$. 
	Then
	\begin{equation}\label{eq:Sq-theta}
		S_q(\lambda)
		=\lambda_{\max}^{q-1}\left(\frac q2 \lambda_{\min}+\left(1-\frac q2\right)\lambda_{\max}\right)
		=\lambda_{\max}^q\left(1-q\frac{\lambda_{\max}- \lambda_{\min}}{2\lambda_{\max}}\right).
	\end{equation}
	Thus the largest value of $q$ for which $\lambda \in \Gamma_q$ is
	\[
	q_\lambda:=
	\frac{2\lambda_{\max}}{\lambda_{\max}-\lambda_{\min}} \in (0,\infty],
	\]
	where the quotient is understood as $\infty$ when
	$\lambda_{\max}=\lambda_{\min}$.
	If $q_\lambda<\infty$, then $S_{q_\lambda}(\lambda)=0$; hence
	$\log S_q(\lambda)$ is defined on $0\le q<q_\lambda$, with the endpoint
	understood by a limiting argument.
	
	\begin{proposition}\label{prop:q-variation}
		For fixed $\lambda$ with $\lambda_{\max}>0$, the function
		$q\mapsto\log S_q(\lambda)$ is concave on $[0,q_\lambda)$. Moreover,
		the function
		\[
		q\longmapsto S_q(\lambda)^{1/q}
		\]
		is nonincreasing on $(0,q_\lambda)$, and is strictly decreasing there
		unless $\lambda_1=\lambda_2$.
	\end{proposition}
	
	\medskip\noindent\emph{Proof.}
	If $\lambda_{\max}=\lambda_{\min}$, then
	$S_q(\lambda)=\lambda_{\max}^q$ and both assertions are immediate.
	For $\lambda_{\max}>\lambda_{\min}$, equation
	\eqref{eq:Sq-theta} gives
	\[
	\log S_q(\lambda)
	=q\log\lambda_{\max}+\log(1-q\theta),
	\qquad
	\theta:=\frac{\lambda_{\max}-\lambda_{\min}}
	{2\lambda_{\max}}>0.
	\]
	Consequently,
	\[
	\frac{d^2}{dq^2}\log S_q(\lambda)
	=-\frac{\theta^2}{(1-q\theta)^2}<0,
	\]
	which proves strict concavity. Since $S_0(\lambda)=1$,
	\[
	\log\bigl(S_q(\lambda)^{1/q}\bigr)
	=\frac{\log S_q(\lambda)}q
	=\frac{\log S_q(\lambda)-\log S_0(\lambda)}{q-0}
	\]
	is the secant slope of the strictly concave function
	$q\mapsto\log S_q(\lambda)$ between $0$ and $q$. These secant slopes
	are strictly decreasing as the right endpoint $q$ increases. Hence
	$S_q(\lambda)^{1/q}$ is strictly decreasing unless
	$\lambda_1=\lambda_2$. \qed
	
	\medskip
	The proposition is the continuous analogue of the familiar Newton and Maclaurin
	inequalities for elementary symmetric polynomials. In its normalized form,
	Newton's inequality says
	\[
	S_k^2
	\geq
	S_{k-1} S_{k+1}.
	\]
	Thus the normalized sequence is log-concave in its index. Maclaurin's
	inequality says that the normalized elementary means
	$S_k^{1/k}$ decrease with $k$.
	Here the continuous
	parameter $q$ gives the corresponding statements
	\[
	S_q(\lambda)^2\geq S_{q-h}(\lambda)S_{q+h}(\lambda),
	\qquad
	S_q(\lambda)^{1/q}\ \text{is nonincreasing in }q.
	\]

	Also, note that when
	$\lambda\in\Gamma_2$, $S_q(\lambda)^{1/q}$ at $q=1$ and $q=2$ are respectively the
	arithmetic and geometric means of $\lambda_1$ and $\lambda_2$, so the family $S_q(\lambda)^{1/q}$ 
	interpolates monotonically between those means; Proposition~\ref{prop:q-variation} also provides a calculus proof of the inequality between arithmetic and geometric means along with the equality cases.
	
	\subsection{Ellipticity and concavity of the equation.} \label{sec:elliptic}
	For $q>1$, set
	\[
	\mathcal F_q(H):=S_q(H)^{1/q}
	\]
	for $H\in\Gamma_q$. We verify directly that $\mathcal F_q$ is concave
	on $\Gamma_q$ for every $q>1$, and that it is also degenerate elliptic
	when $1<q\leq2$. Let
	\[
	M=\lambda_{\max}(H),\qquad m=\lambda_{\min}(H),\qquad
	Y=\frac q2m+\left(1-\frac q2\right)M.
	\]
	Then
	\[
	\mathcal F_q(H)=M^{(q-1)/q}Y^{1/q}.
	\]
	
	Suppose first that $1<q\leq2$. In the interior of $\Gamma_q$,
	\[
	\partial_m\mathcal F_q
	=\frac12 M^{(q-1)/q}Y^{1/q-1}>0,
	\]
	while
	\[
	\partial_M\mathcal F_q
	=\frac{q-1}{q}M^{-1/q}Y^{1/q}
	+\frac1q\left(1-\frac q2\right)
	M^{(q-1)/q}Y^{1/q-1}>0.
	\]
	Hence $\mathcal F_q$ is increasing in both eigenvalues. Weyl
	monotonicity \cite[Corollary~9.9]{Tropp} gives
	\[
	\mathcal F_q(H+P)\geq\mathcal F_q(H)
	\]
	for $P\geq0$ whenever both matrices lie in $\Gamma_q$. By continuity
	this holds on the closed cone, so $\mathcal F_q$ is degenerate
	elliptic, and is strictly elliptic in the interior.
	
	For any $q>1$, we prove concavity directly in the unordered eigenvalues
	$(\lambda_1,\lambda_2)$. On the region $\lambda_1\geq\lambda_2$,
	\[
	\mathcal F_q(\lambda_1,\lambda_2)
	=\lambda_1^{(q-1)/q}
	\left(\frac q2\lambda_2+
	\left(1-\frac q2\right)\lambda_1\right)^{1/q}.
	\]
	This is the weighted geometric mean of two affine functions, and hence
	is concave. The same argument applies on the region
	$\lambda_2\geq\lambda_1$. Along their common boundary
	$\lambda_1=\lambda_2=h>0$, direct differentiation from either side
	gives
	\[
	\nabla\mathcal F_q(h,h)=\left(\frac12,\frac12\right).
	\]
	Thus the two concave pieces have matching first derivatives and glue
	to a concave symmetric function on the eigenvalue cone $\Gamma_q$.
	Continuity gives the same conclusion on its boundary. The standard
	spectral function criterion now implies that
	$H\mapsto\mathcal F_q(H)$ is concave on $\Gamma_q$.
	Thus $\mathcal F_q$ is concave on $\Gamma_q$ for every $q>1$. For
	$1<q\leq2$, it is in addition degenerate elliptic, so
	$\mathcal F_q(D^2u)=f$ is a concave, degenerate elliptic fully nonlinear
	PDE on the stated cone.


	\section{Proof of Proposition~\ref{prop:dual-burkholder} from Theorem~\ref{thm:local-burkholder}}
	\label{sec:strain}
	Let $1<q\leq2$, let $p=q'$, and let $B$ be as in
	Proposition~\ref{prop:dual-burkholder}. By
	\eqref{eq:Gamma_q^*Gamma_{q'}}, $B\in \Gamma_q^* = \Gamma_p$, and
	\eqref{eq:rho-eigenvalue-formula} gives
	\begin{equation}\label{eq:rho-burkholder-identification}
		\rho_q^*(B)^p
		=S_p(B).
	\end{equation}
	Thus our goal is to control the integral of $S_p(B)$ on $\D$.
	
	\subsection{Symmetric-gradient representation and boundary control}
	
	Let $E = \cof B$ be the cofactor matrix of $B$.
	Since $p\geq2$, every matrix in $\Gamma_p$ is positive semidefinite.
	The matrices $B$ and $E$ have the same eigenvalues, and therefore
	$E\in\Gamma_p$ and
	\[
	S_p(E)=S_p(B).
	\]
	We therefore seek to control $S_p(E)$. To apply
	Theorem~\ref{thm:local-burkholder}, we would like to realize $E$ as
	the symmetric gradient of a vector field. To check whether this is possible, we compute
	\[
	\inc E:=\partial_{22}E_{11}+\partial_{11}E_{22}-2\partial_{12}E_{12}.
	\]
	Here $\inc$ denotes the incompatibility operator for symmetric
	gradients. From $E = \cof B$ we have
	\[
	\inc E=\operatorname{div}\operatorname{div}B=-\delta_0,
	\]
	so the point mass singularity prevents the representation of $E$ as a symmetric gradient on the
	full disk $\D$. To cancel it, put $\nu(x)=x/|x|$ for $x\ne0$ and introduce
	the fixed bounded strain
	\[
	E_0(x)=\frac{1}{2\pi}\bigl(I-\nu(x)\otimes\nu(x)\bigr).
	\]
	Its value at $0$ is immaterial, and $\cof E_0=\frac{1}{2\pi}\nu\otimes\nu$.
	Direct differentiation away from $0$, followed by the flux computation
	at $0$, gives
	\[
	\operatorname{div}(\nu\otimes\nu)=\frac{x}{|x|^2},\qquad
	\operatorname{div}\frac{x}{|x|^2}=2\pi\delta_0.
	\]
	Consequently
	\[
	\inc E_0=\delta_0,\qquad \inc(E+E_0)=0\quad\text{in }\D.
	\]

	The $L^2$ Saint--Venant theorem on a simply connected disk says that
	$S\in L^2(\D;\Sym(2))$ with $\inc S=0$ is the symmetric gradient of a vector field
	$W\in W^{1,2}(\D;\R^2)$, unique up to the addition of a field
	$x\mapsto cJx+v$, where $c\in\R$ and $v\in\R^2$;
	see \cite{ACGK}.
	In the plane its two integrations are particularly transparent: since $\inc S = 0$, one can solve for a scalar $\omega$ up to a constant by requiring
	\begin{equation}\label{eq:rotation-equations}
		\partial_1\omega=\partial_1S_{12}-\partial_2S_{11},\qquad
		\partial_2\omega=\partial_1S_{22}-\partial_2S_{12}.
	\end{equation}
	Recall the almost complex structure $J$ from \eqref{eq:SymJ}. The rows of $S+\omega J$ are then curl-free and integrate to $W$. Since $J$ is skew-symmetric,
	$\sym(\omega J)=0$, so $\sym DW=\sym(S+\omega J)=S$.
	For $L^2$ tensors, one may mollify on an increasing family of smaller disks. On each disk the curl-free rows give $W_k$ with $\sym DW_k=S_k$; after subtracting a field $x\mapsto c_kJx+v_k$, Korn inequality gives
	\[
	\|W_k\|_{W^{1,2}}\leq C\|S_k\|_{L^2}.
	\]
	Weak compactness and exhaustion of $\D$ by smaller disks then produce $W\in W^{1,2}(\D,\R^2)$ with $\sym DW=S$.
	If $S$ is smooth up to a boundary arc, the identities \eqref{eq:rotation-equations} also show that $W$ is smooth up to that arc.
	
	Apply this fact to $S=E+E_0$ to obtain one vector field
	\[
	W\in W^{1,2}(\D;\R^2),\qquad \sym DW=E+E_0.
	\]
	
	The point of this construction is that, for a sufficiently large constant
	$a=a(q)$, the map $Z(x):=ax+W(x)$ satisfies
	\begin{equation}\label{eq:early-strain-comparison}
		S_p(B)=S_p(E)\leq S_p(DZ).
	\end{equation}
	This will be proved below using the trace and eigenvalue gap comparisons
	from Section~\ref{sec:Sq}. Thus it remains only to control the integral 
	of $S_p(DZ)$ over $\D$. Theorem~\ref{thm:local-burkholder} does not apply directly,
	because the boundary values of $Z$ on $\partial \D$ need not be linear. The next lemma
	provides the boundary control needed using a cutoff that corrects them.
	\begin{lemma}\label{lem:boundary}
		After adding to $W$ a suitable affine map $x\mapsto cJx+v$, where
		$c\in\R$ and $v\in\R^2$, the periodic boundary function
		$h(\theta)=W(e(\theta))$, where
		$e(\theta)=(\cos\theta,\sin\theta)$, satisfies
		\begin{equation}\label{eq:boundary-control}
			\|h\|_{L^\infty(\R/2\pi\mathbb Z)}+
			\|h'\|_{L^\infty(\R/2\pi\mathbb Z)}\leq C,
		\end{equation}
		with an absolute constant independent of $B$.
	\end{lemma}
	
	\begin{proof}
		
		Let $h(\theta) = W(e(\theta))$. The chain rule gives $h'(\theta)=DW(e(\theta))e'(\theta) = DW(e(\theta)) Je(\theta)$. 
		Write
		\[
		DW=E+E_0+\omega J,
		\qquad \omega=\tfrac12(\partial_1W_2-\partial_2W_1).
		\]
		We use the same letter
		$\omega(\theta)$ for its restriction to that circle. $E=0$ on $\partial \D$ and $E_0=(2\pi)^{-1}(Je)\otimes(Je)$, while $J(Je)=-e$. Therefore our earlier formula for $h'$ gives
		\begin{equation}\label{eq:boundary-derivative}
			h'(\theta)=\frac{1}{2\pi}Je(\theta)-\omega(\theta)e(\theta).
		\end{equation}
		We claim that
		\begin{equation}\label{eq:rotation-bound}
			\omega'(\theta)\leq \frac{1}{2\pi}\quad\text{for every }\theta.
		\end{equation}
		By simultaneously rotating the coordinates and vector components,
		it suffices to calculate at $e(0)=(1,0)$. There
		\eqref{eq:rotation-equations} yields
		\[
		\omega'(0)=\partial_2\omega
		=\partial_1(E_{22}+(E_0)_{22})-\partial_2(E_{12}+(E_0)_{12}).
		\]
		Since $E_{12}(e(\theta))=0$ on the boundary, differentiating at $\theta=0$ gives $\partial_2E_{12}(1,0)=0$.
		Since $E\geq0$ inside $\D$ and $E_{22}(1,0)=0$, its outward derivative
		satisfies $\partial_1E_{22}(1,0)\leq0$. Finally,
		\[
		(E_0)_{22}=\frac{1}{2\pi}\frac{x_1^2}{|x|^2},\qquad
		(E_0)_{12}=-\frac{1}{2\pi}\frac{x_1x_2}{|x|^2}
		\]
		give $\partial_1(E_0)_{22}(1,0)=0$ and
		$\partial_2(E_0)_{12}(1,0)=-\frac{1}{2\pi}$. This proves \eqref{eq:rotation-bound}.
		
		The function $\omega(\theta)$ is single-valued and $2\pi$-periodic.
		Integrating its one-sided derivative bound from a minimum to the next maximum counterclockwise shows that
		\[
		\max\omega-\min\omega\leq2\pi \cdot \frac{1}{2\pi}=1.
		\]
		Subtracting a suitable constant multiple of $Jx$ from $W$, which does
		not change its symmetric gradient, makes $|\omega|\leq1/2$ on the
		whole circle. Equation~\eqref{eq:boundary-derivative} now bounds $h'$.
		Subtracting by $v := W(e(0)) \in \R^2$ makes $h(0)=0$ so from integrating the bound on $h'$ around the circle, we bound $h$ in $L^{\infty}(\R/2\pi \mathbb Z)$ as well.
	\end{proof}

	\subsection{The affine correction and the radial cutoff}
	Take $W$ with the normalization in Lemma~\ref{lem:boundary}, and set
	\[
	Z(x)=ax+W(x),\qquad x\in\D,
	\]
	where $a=a(q)>0$ will be chosen below. Its boundary values are
	$Z(e(\theta))=ae(\theta)+h(\theta)$.
	
	In $\D$ we have
	\[
	\sym DZ=aI+E+E_0.
	\]
	Since $\operatorname{tr}E_0=1/(2\pi)$ and
	$\gap(E_0)=1/(2\pi)$, we have (from \eqref{eq:gaptriangle})
	\[
	\operatorname{tr}(\sym DZ)=
	2a+\operatorname{tr}(E)+\frac1{2\pi},\qquad
	\gap(\sym DZ)\leq
	\gap(E)+\frac1{2\pi}.
	\]
	Consequently,
	\begin{align}
		&\operatorname{tr}(\sym DZ)
		-(p-1)\gap(\sym DZ) \geq
		\operatorname{tr}(E)-(p-1)\gap(E)
		+2a-\frac{p-2}{2\pi}. \label{eq:strain-gap}
	\end{align}
	Since $E\in\Gamma_p$, equation~\eqref{eq:primal-trace-gap} shows that
	\(\operatorname{tr}(E)-(p-1)\gap(E)\geq0\). Choose $a=a(q)$ so that
	the last two terms in \eqref{eq:strain-gap} have nonnegative sum.
	Then \eqref{eq:primal-trace-gap} gives $\sym DZ\in\Gamma_p$.
	If $M,m$ are the
	signed singular values of $DZ$, \eqref{eq:singular-sum}--\eqref{eq:singular-gap}
	and \eqref{eq:strain-gap} give
	\[
	(M+m)-(p-1)(M-m)\geq
	\operatorname{tr}(E)-(p-1)\gap(E).
	\]
	Moreover,
	\[
	M=\|DZ\|_{\mathrm{op}}\geq
	\lambda_{\max}(\sym DZ)\geq
	\lambda_{\max}(E)=\frac12\bigl(\operatorname{tr}(E)+\gap(E)\bigr).
	\]
	The factorization in \eqref{eq:Sp-trace-gap}, together with the same
	algebra applied to the signed singular values of $DZ$, gives
	\begin{align*}
		S_p(DZ)
		&=\frac1{2^p}(2M)^{p-1}
		\bigl((M+m)-(p-1)(M-m)\bigr),\\
		S_p(E)
		&=\frac1{2^p}\bigl(\operatorname{tr}(E)+\gap(E)\bigr)^{p-1}
		\bigl(\operatorname{tr}(E)-(p-1)\gap(E)\bigr).
	\end{align*}
	Both final factors are nonnegative, so the preceding two inequalities
	show that
	\begin{equation}\label{eq:strain-dominates-dual-density}
		S_p(DZ)\geq S_p(E)=S_p(B)=\rho_q^*(B)^p\geq0
		\qquad\text{in }\D.
	\end{equation}
	
	Theorem~\ref{thm:local-burkholder} cannot yet be applied: the boundary
	values $ax+h$ of $Z$ are not in general the restriction of a linear map.
	To correct them, choose $\chi\in C^\infty([1,2])$ equal to $1$ near
	$1$ and to $0$ near $2$, and define $F$ on $B_2(0)$ by
	\[
	F(x)=
	\begin{cases}
		Z(x),&|x|<1,\\
		ax+\chi(|x|)h(\theta),&x=|x|e(\theta),\quad1\leq|x|<2.
	\end{cases}
	\]
	The traces agree at $|x|=1$, and the correction vanishes near $|x|=2$.
	Thus $F-ax\in W^{1,2}_0(B_2(0);\R^2)$. By
	\eqref{eq:boundary-control}, the derivative of the annular correction
	is bounded by an absolute constant. On $B_2(0)\setminus\D$, write
	$DF=aI+P$, where $|P|\leq C$. The estimates \eqref{eq:singular-sum}--\eqref{eq:singular-gap} give
	\[
	M(DF)+m(DF)\geq2a-C,\qquad M(DF)-m(DF)\leq C.
	\]
	Increasing $a(q)$ if necessary, we obtain $S_p(DF)\geq0$ throughout
	$B_2(0)$. Applying Theorem~\ref{thm:local-burkholder} on $B_1(0)$ to
	$y\mapsto F(2y)/2$, whose boundary values are $ay$, and changing
	variables, we obtain
	\begin{equation}\label{eq:burkholder-integral}
		\int_{B_2(0)}S_p(DF)\,dx\leq4\pi a^p\leq C_q.
	\end{equation}
	Finally, \eqref{eq:strain-dominates-dual-density} gives
	\[
	\int_\D\rho_q^*(B)^p\,dx
	\leq\int_\D S_p(DF)\,dx
	\leq\int_{B_2(0)}S_p(DF)\,dx
	\leq C_q,
	\]
	which proves Proposition~\ref{prop:dual-burkholder}.
	
	\begin{remark}[Scaling the dual Pucci inequality]
		If the pole is $x_0$ and the disk has radius $R$, set
		$\widetilde B(y)=B(x_0+Ry)$. In dimension two,
		$\operatorname{div}_y\operatorname{div}_y\widetilde B=-\delta_0$
		without any amplitude rescaling: the second-derivative factor cancels
		the scaling of the point mass. Hence the proposition gives
		\begin{equation}\label{eq:dual-pucci-scaling}
			\int_{B_R(x_0)}\rho_q^*(B)^{q'}\dx\leq C_q R^2.
		\end{equation}
	\end{remark}
	
	\section{Proof of Theorem~\ref{thm:weighted-maximum}}\label{sec:pf_thm1}
	
	%
	We use the standard Dirichlet and adjoint Green theory on a smooth disk
	for a smooth uniformly elliptic nondivergence operator
	$L=A_{ij}\partial_{ij}$; see \cite{GT,DK}. With the sign convention
	$-Lv=g$, $v|_{\partial B_R}=0$, its positive adjoint Green kernel
	$G(x_0,\cdot)$ satisfies
	\begin{equation}\label{eq:green-input}
		-\partial_i\partial_j(A_{ij}G)=\delta_{x_0},\qquad
		G=0\ \text{on }\partial B_R,\qquad
		v(x_0)=\int_{B_R}G(x_0,x)g(x)\dx.
	\end{equation}
	For smooth coefficients, $G$ is smooth up to the boundary away from
	$x_0$ and belongs to $L^2(B_R)$ in dimension two. One can obtain the
	last assertion from the classical $W^{2,2}$ Dirichlet estimate and
	the planar Sobolev embedding: they bound $g\mapsto v(x_0)$ on $L^2$,
	with a constant allowed to depend on $A$. Positivity of the Dirichlet
	solution operator and the Riesz representation theorem then give an
	$L^2$ kernel. The distributional
	identity follows by testing with $g=-L\phi$, and standard adjoint
	regularity and the Dirichlet Green identity give the boundary and
	smoothness assertions.
	
	\begin{proof}[Proof of Theorem~\ref{thm:weighted-maximum}]
		Our goal is to prove \eqref{eq:weighted-maximum}. Dividing $A$ and $f$ by $\tr A$ preserves their weighted quotient. Thus we may assume $\tr A=1$ and $|A|\leq1$ a.e. We may also assume $f^- / \rho_q^*(A) \in L^q(\Omega)$.
		
		\medskip
		\emph{Step 1: smooth coefficients and solutions.}
		First let $A$ and $u$ be smooth on a neighborhood of
		$\overline\Omega$, with $A\in\operatorname{int}\Gamma_q^*$ there.
		Set $D=\diam\Omega$ and fix $x_0\in\Omega$. Then $\overline\Omega\Subset B_D(x_0)$.
		Choose an open set $U$ with
		$\overline\Omega\subset U\Subset B_D(x_0)$ on which $A$ is smooth and
		takes values in $\operatorname{int}\Gamma_q^*$, and choose
		$\chi\in C_c^\infty(U)$ such that $0\leq\chi\leq1$ and
		$\chi=1$ on a neighborhood of $\overline\Omega$. Define
		\[
		\widetilde A=\frac I2+\chi\left(A-\frac I2\right)\quad\text{in }U,
		\qquad \widetilde A=\frac I2\quad\text{in }B_D(x_0)\setminus U.
		\]
		The trace-one slice of $\operatorname{int}\Gamma_q^*$ is convex, so
		\[
		\tr\widetilde A=1,\qquad
		\lambda_{\min}(\widetilde A)>1-q/2,\qquad
		\widetilde A=I/2\quad\text{near }\partial B_D(x_0).
		\]
		In particular, $\widetilde A$ is smooth and uniformly elliptic on the closed
		disk. We henceforth write $A$ for this extension.
		Let $f=\tr(AD^2u)$ in $\Omega$ and solve
		\[
		-\tr(AD^2v)=f^-\mathbf1_\Omega\quad\text{in }B_D(x_0),
		\qquad v=0\quad\text{on }\partial B_D(x_0).
		\]
		The solution is nonnegative and belongs to $W^{2,s}$ for every finite
		$s>1$. Put $b=\sup_{\partial\Omega}u$. On $\Omega$, the function
		$v+b-u$ has nonnegative boundary values and satisfies
		$\tr(A D^2(v+b-u))=-f^--f=-f^+\leq0$. The minimum principle
		therefore gives
		\begin{equation}\label{eq:green-comparison}
			u(x_0)\leq b+v(x_0)
			=b+\int_\Omega G(x_0,x)f^-(x)\dx.
		\end{equation}
		The tensor $B=GA$ has all the properties in
		Proposition~\ref{prop:dual-burkholder}, after dilation: $B\in L^2$,
		it is smooth away from $x_0$, vanishes on the outer circle, and
		$\operatorname{div}\operatorname{div}B=-\delta_{x_0}$.
		Using $\rho_q^*(GA)=G\rho_q^*(A)$ and \eqref{eq:dual-pucci-scaling},
		\[
		\int_{B_D(x_0)}G^{q'}\rho_q^*(A)^{q'}\dx\leq C_qD^2.
		\]
		H\"older inequality gives
		\begin{equation}\label{eq:green-holder}
			\int_\Omega G(x_0,x)f^-(x)\,dx
			\leq \|G\rho_q^*(A)\|_{L^{q^{\prime}}(B_D(x_0))}
			\left\|\frac{f^-}{\rho_q^*(A)}\right\|_{L^q(\Omega)}
			\leq C_qD^{2/q^{\prime}}
			\left\|\frac{f^-}{\rho_q^*(A)}\right\|_{L^q(\Omega)}.
		\end{equation}
		Combining \eqref{eq:green-comparison} and \eqref{eq:green-holder}, and using $2/q^{\prime}=2-2/q$, proves \eqref{eq:weighted-maximum} in the smooth case. The ellipticity and smoothness constants of $A$ were used
		only to construct $G$, not in the final bound.
		
		\medskip
		\emph{Step 2: measurable coefficients and Sobolev solutions on compact subdomains.}
		Let $V\Subset\Omega$ be open. Extend the trace-normalized $A$ by
		$I/2$ outside $\Omega$ and, for $0<\varepsilon<1$, set
		\[
		A_\varepsilon=(1-\varepsilon)A+\varepsilon I/2,
		\qquad f_\varepsilon=\tr(A_\varepsilon D^2u).
		\]
		Mollify $A_\varepsilon$ to obtain smooth fields $A_{\varepsilon,j}$
		converging a.e. to $A_\varepsilon$. Convexity of the cone and
		\eqref{eq:rho-regularization} give
		\[
		\tr A_{\varepsilon,j}=1,\qquad |A_{\varepsilon,j}|\leq1,
		\qquad \rho_q^*(A_{\varepsilon,j})\geq\varepsilon/2.
		\]
		Indeed, each mollification is $(1-\varepsilon)$ times a
		$\Gamma_q^*$-valued average plus $\varepsilon I/2$.
		Mollify $u$ on a neighborhood of $\overline V$ so that
		$u_j\to u$ uniformly there and $D^2u_j\to D^2u$ in $L^q(V)$.
		For fixed $\varepsilon$, continuity of $\rho_q^*$ and the preceding
		lower bound imply
		\begin{equation}\label{eq:smooth-quotient-limit}
			\frac{\tr(A_{\varepsilon,j}D^2u_j)}{\rho_q^*(A_{\varepsilon,j})}
			\longrightarrow
			\frac{f_\varepsilon}{\rho_q^*(A_\varepsilon)}
			\quad\text{in }L^q(V).
		\end{equation}
		To check this, the part containing $D^2u_j-D^2u$ is bounded in
		absolute value by $2\varepsilon^{-1}|D^2u_j-D^2u|$; the remaining
		part converges a.e. and is dominated by a constant times
		$\varepsilon^{-1}|D^2u|$.
		Since the map $s\mapsto s^-$ is $1$-Lipschitz, the same convergence
		holds after taking negative parts.
		Apply Step 1 on $V$ to $(A_{\varepsilon,j},u_j)$ and let $j\to\infty$.
		
		We next let $\varepsilon\downarrow0$. Since
		$f_\varepsilon=(1-\varepsilon)f+(\varepsilon/2)\Delta u$,
		\eqref{eq:rho-regularization} gives
		\begin{equation}\label{eq:quotient-domination}
			\frac{f_\varepsilon^-}{\rho_q^*(A_\varepsilon)}
			\leq \frac{f^-}{\rho_q^*(A)}+(\Delta u)^-.
		\end{equation}
		The right-hand side is in $L^q(V)$ and
		$\rho_q^*(A)>0$ a.e. Since
		$f_\varepsilon/\rho_q^*(A_\varepsilon)$ converges pointwise to
		$f/\rho_q^*(A)$, dominated convergence gives
		\begin{equation}\label{eq:local-boundary-estimate}
			\sup_V u
			\leq \sup_{\partial V}u+C_q\,\diam(V)^{2-2/q}
			\left\|\frac{f^-}{\rho_q^*(A)}\right\|_{L^q(V)}.
		\end{equation}
		
		\medskip
		\emph{Step 3: localization at the boundary.}
		Let $b=\sup_{\partial\Omega}u$ and
		$V_\eta=\{x\in\Omega:u(x)>b+\eta\}$ for $\eta>0$.
		If nonempty, $V_\eta\Subset\Omega$ and $u=b+\eta$ on its boundary.
		Apply \eqref{eq:local-boundary-estimate} on $V_\eta$, whose diameter
		and source norm are bounded by those on $\Omega$, and let
		$\eta\downarrow0$. This proves \eqref{eq:weighted-maximum} and hence
		the theorem.
	\end{proof}
	
	\begin{remark}
		Applying Theorem~\ref{thm:weighted-maximum} to $u$ and $-u$ gives the
		two-sided consequence
		\begin{equation}\label{eq:weighted-maximum-two-sided}
			\|u\|_{L^\infty(\Omega)}
			\leq \max_{\partial\Omega}|u|
			+C_q\,\diam(\Omega)^{2-2/q}
			\left\|\frac{f}{\rho_q^*(A)}\right\|_{L^q(\Omega)}.
		\end{equation}
	\end{remark}
	
	\section{Proof of the Sobolev estimate}
	\label{sec:hessian-proof}
	
	\begin{proof}[Proof of Theorem~\ref{thm:Sobolev}]	
		By Section~\ref{sec:Sq}, the positively homogeneous function
		\(\mathcal F_q=S_q^{1/q}\) is concave and degenerate elliptic on
		\(\Gamma_q\), and strictly elliptic in its interior. It is also
		\(C^1\) in the interior. The only point requiring comment is a scalar
		matrix $H=hI$, where $h>0$ and the eigenvalues coincide; the eigenvalue
		formula gives $D\mathcal F_q(hI)=I/2$.
		
		For \(H\in\operatorname{int}\Gamma_q\), set
		\[
		A(H)=D\mathcal F_q(H).
		\]
		Euler's identity gives
		\[
		\tr(A(H)H)=\mathcal F_q(H).
		\]
		By concavity,
		\[
		\mathcal F_q(K)
		\leq \mathcal F_q(H)+\tr\bigl(A(H)(K-H)\bigr)
		=\tr(A(H)K) 
		\qquad \text{for } K\in\Gamma_q.
		\]
		Consequently \(A(H)\in\Gamma_q^*\). In the defining
		infimum for \(\rho_q^*(A(H))\), this inequality shows that every
		quotient is at least $1/2$, while the choice $K=H$ and Euler's identity
		give equality. Hence
		\begin{equation}\label{eq:calibration}
			\rho_q^*(A(H))=\frac12,
		\end{equation}
		and in particular \(A(H)\in\operatorname{int}\Gamma_q^*\).
		
		Let $H=D^2u\in\Gamma_q$ a.e. For $0<\varepsilon<1$, use the
		measurable fields
		\[
		A_\varepsilon=D\mathcal F_q(H+\varepsilon I),
		\qquad f_\varepsilon=\tr(A_\varepsilon H).
		\]
		The shift lies in the interior, and \eqref{eq:calibration} gives
		$\rho_q^*(A_\varepsilon)=1/2$. Since $A_\varepsilon\geq0$, the
		supporting inequality and Euler's identity imply
		\begin{equation}\label{eq:closed-cone-calibration}
			0\leq\mathcal F_q(H)\leq f_\varepsilon
			=\mathcal F_q(H+\varepsilon I)-\varepsilon\tr A_\varepsilon
			\leq\mathcal F_q(H+\varepsilon I)\leq |H|+\varepsilon.
		\end{equation}
		The last bound follows from $\mathcal F_q(K)\leq M(K)$ on $\Gamma_q$.
		On every $V\Subset\Omega$, it follows that
		$f_\varepsilon\to\mathcal F_q(D^2u)$ in $L^q(V)$.
		Apply \eqref{eq:weighted-maximum-two-sided} on $V$ to $A_\varepsilon$
		and $u$, then let $\varepsilon\downarrow0$. After absorbing the factor
		$2$,
		\begin{equation}\label{eq:hessian-boundary-estimate}
			\|u\|_{L^\infty(V)}
			\leq\max_{\partial V}|u|
			+C_q\,\diam(V)^{2-2/q}
			\left(\int_V S_q(D^2u)\dx\right)^{1/q}.
		\end{equation}
		Take $V=\{|u|>\eta\}\Subset\Omega$ and let $\eta\downarrow0$ to
		obtain \eqref{eq:target}.
		
		%
		%
	\end{proof}
	
	\begin{remark} The conclusion \eqref{eq:target} fails if we only assume $u\in W^{2,s}_{\loc}$ for every $s<q$.
		For $1<q<2$, set $\beta=2-2/q$ and
		$u(x)=|x|^\beta-1$ on $\D$. This function is continuous on
		$\overline\D$, vanishes on the circle, and belongs to
		$W^{2,s}_{\loc}$ for every $1\leq s<q$, but not to $W^{2,q}_{\loc}$.
		For $x\ne0$, its Hessian has eigenvalues
		\[
		M=\beta|x|^{\beta-2},\qquad m=(\beta-1)M,
		\]
		so $D^2u\in\partial\Gamma_q$ and $S_q(D^2u)=0$ a.e. Nevertheless,
		$\|u\|_\infty=1$. Hence \eqref{eq:target}, with only the a.e. density
		on its right-hand side, is false under any of these lower Sobolev
		assumptions. The same failure of the a.e.-density formulation occurs
		at $q=2$: the function $|x|-1$ belongs to $W^{2,s}_{\loc}$ for every
		$s<2$ and satisfies $\det D^2u=0$ a.e. However, its Monge--Amp\`ere
		measure is $\pi\delta_0$, so the measure estimate
		\eqref{eq:MA-target} retains the missing mass. Section~\ref{sec:measures}
		shows that no analogous weakly continuous $q$-Hessian measure is
		available when $1<q<2$.
	\end{remark}
	
	\section{Proof of Theorem~\ref{thm:local-burkholder}} \label{sec:burkholder}
	
	We first dispose of the case $p=2$. If $A\in\R^{2\times2}$ and
	$\varphi\in W^{1,2}_0(\D;\R^2)$, then
	\begin{equation}\label{eq:moments-disk}
		\int_\D\det(A+D\varphi)\,dx=|\D|\det A.
	\end{equation}
	Indeed, for smooth compactly supported $\varphi$,
	\[
	\begin{split}
		\det(A+D\varphi)
		&=\det A+\sum_{i,j=1}^2(\cof A)_{ij}\partial_j\varphi^i
		+\det D\varphi,\\
		\det D\varphi
		&=\partial_1(\varphi^1\partial_2\varphi^2)
		-\partial_2(\varphi^1\partial_1\varphi^2).
	\end{split}
	\]
	Thus the identity follows by integration. Approximation in $W^{1,2}$
	gives the stated result, since
	\[
	|\det P-\det\widetilde P|
	\leq(|P|+|\widetilde P|)|P-\widetilde P|.
	\]
	As $S_2=\det$, this proves Theorem~\ref{thm:local-burkholder} when
	$p=2$. We henceforth assume $p>2$.
	
	For $t>0$, define
	\begin{equation}\label{eq:Lt-definition}
		L_t(P):=\det P-(M(P)-t)_+(m(P)-t).
	\end{equation}
	The integrand $L_1$ was introduced independently by
	Burkholder and \v{S}ver\'ak \cite{Burkholder,Sverak}, and $L_t(P)$ is just $t^2 L_1(P/t)$.
	See also \cite[\S1]{Guerra} and \cite[\S1]{BMS} for the history and the different normalizations.
	The following inequality is the main property of $L_t$ that we need.
	
	\begin{proposition}[Integral inequality for $L_t$]\label{prop:cutoff}
		For every $t>0$, $A\in\R^{2\times2}$, and
		$f-Ax\in W^{1,2}_0(\D;\R^2)$,
		\begin{equation}\label{eq:cutoff-integral}
			\int_\D L_t(Df)\,dx\geq|\D|L_t(A).
		\end{equation}
	\end{proposition}
	Under the hypotheses of Proposition~\ref{prop:cutoff}, it follows that
	\begin{equation}\label{eq:remainder-integral}
		\int_\D(M(Df)-t)_+(m(Df)-t)\,dx
		\leq|\D|(M(A)-t)_+(m(A)-t).
	\end{equation}
	Indeed, \eqref{eq:remainder-integral} follows immediately by subtracting
	\eqref{eq:cutoff-integral} from \eqref{eq:moments-disk} and using
	\[
	\det P-L_t(P)=(M(P)-t)_+(m(P)-t).
	\]
	
	\begin{proof}[Proof of Theorem~\ref{thm:local-burkholder}]
		For every finite $a$, Fubini's theorem applies because
		\[
		\int_\D\int_0^a t^{p-3}
		|(M(Df)-t)_+(m(Df)-t)|\,dt\,dx
		\leq\frac{2a^{p-2}}{p-2}\int_\D M(Df)^2\,dx<\infty.
		\]
		Multiplying \eqref{eq:remainder-integral} by
		$\frac12p(p-1)(p-2)t^{p-3}$ and integrating over $(0,a)$ gives
		\[
		\int_\D H_{p,a}(Df)\,dx\leq|\D|H_{p,a}(A).
		\]
		By \eqref{eq:positive-truncations}, the hypothesis $S_p(Df)\geq0$
		a.e. makes the left-hand integrands nonnegative.
		Equation~\eqref{eq:truncation-equals-Sp}
		gives $H_{p,a}(Df(x))\to S_p(Df(x))$ a.e., while
		$H_{p,a}(A)=S_p(A)$ for $a\geq M(A)$. Fatou's lemma therefore yields
		\[
		0\leq\int_\D S_p(Df)\,dx
		\leq\liminf_{a\to\infty}\int_\D H_{p,a}(Df)\,dx
		\leq|\D|S_p(A).
		\]
		This proves both the asserted integrability and the inequality.
	\end{proof}
	
	It remains to prove Proposition~\ref{prop:cutoff}. We first record the
	algebraic and projection facts needed for the periodic inequality on
	$Q=(-1,1)^2$. After proving that inequality, we transfer it to the disk
	by zero extension.
	
	\subsection{Pointwise inequalities for \texorpdfstring{$L_t$}{L-t}}
	\label{sec:algebra}
	
	For a matrix $P$, set
	\[
	P=\begin{pmatrix}P_{11}&P_{12}\\P_{21}&P_{22}\end{pmatrix},
	\qquad \mathbf a(P)=(P_{11},P_{12}),\qquad \mathbf b(P)=(P_{22},-P_{21}).
	\]
	\begin{lemma}[Pointwise bounds for $L_t$]\label{lem:pointwise-Lt}
		For every $P\in\R^{2\times2}$,
		\begin{equation}\label{eq:pointwise-cutoff}
			t|\mathbf a(P)+\mathbf b(P)|-t^2\leq L_t(P).
		\end{equation}
		Moreover, if $z\in\R^2$ satisfies either $|z|=t$, or if
		$z=\mathbf a(P)$ and $|\mathbf a(P)|\leq t$, then
		\begin{equation}\label{eq:selection}
			L_t(P)\geq z\cdot\mathbf b(P)+z\cdot(\mathbf a(P)-z).
		\end{equation}
	\end{lemma}
	
	The second assertion is the rescaled, real-matrix form of
	\cite[Lemma~3]{AA}. We include a direct proof of both assertions.
	
	\begin{proof}
		Using $\det P=M(P)m(P)$, we can write \eqref{eq:Lt-definition} as
		\begin{equation}\label{eq:Lt-pieces}
			L_t(P)=
			\begin{cases}
				\det P,&M(P)\leq t,\\[1mm]
				t\bigl(M(P)+m(P)\bigr)-t^2,&M(P)>t.
			\end{cases}
		\end{equation}
		Indeed, when $M(P)>t$,
		\[
		\det P-\bigl(t(M(P)+m(P))-t^2\bigr)
		=(M(P)-t)(m(P)-t),
		\]
		and the two branches agree when $M(P)=t$. Equation
		\eqref{eq:Lt-pieces} will be used in both estimates.
		
		Next,
		\begin{equation}\label{eq:singular-sum-real}
			M(P)+m(P)=|\mathbf a(P)+\mathbf b(P)|.
		\end{equation}
		Indeed,
		\[
		|\mathbf a(P)+\mathbf b(P)|^2
		=|P|^2+2\det P=(M(P)+m(P))^2,
		\]
		and both sides before squaring are nonnegative. Equations
		\eqref{eq:Lt-pieces} and \eqref{eq:singular-sum-real} immediately give
		\eqref{eq:pointwise-cutoff}.
		
		We also have
		\begin{equation}\label{eq:real-identities}
			\det P=\mathbf a(P)\cdot \mathbf b(P),\qquad
			|m(P)|\leq |\mathbf a(P)|\leq M(P).
		\end{equation}
		The determinant identity is immediate, and the bounds follow by applying
		the smallest and largest singular-value bounds to $P^{\mathsf T}e_1$.
		
		It remains to prove \eqref{eq:selection}. If $|z|=t$, then
		\[
		z\cdot(\mathbf a(P)+\mathbf b(P))-t^2
		\leq t|\mathbf a(P)+\mathbf b(P)|-t^2\leq L_t(P).
		\]
		If $z=\mathbf a(P)$, the right-hand side of \eqref{eq:selection} is
		\[
		\mathbf a(P)\cdot\mathbf b(P)
		+\mathbf a(P)\cdot\bigl(\mathbf a(P)-\mathbf a(P)\bigr)
		=\det P.
		\]
		Write $M=M(P)$ and $m=m(P)$. If $M\leq t$, then $\det P=L_t(P)$.
		If $M>t$, then \eqref{eq:real-identities} and $|\mathbf a(P)|\leq t$
		give $m\leq t$, so
		\[
		\det P-L_t(P)=(M-t)(m-t)\leq0.
		\]
		Thus $\det P\leq L_t(P)$ in both cases.
	\end{proof}
	\subsection{The affine gradient-constrained projection}
	\label{sec:projection}
	
	Identify opposite sides of $Q=(-1,1)^2$. The space
	$W^{1,r}_{\per}(Q)$ consists of Sobolev functions that are $2$-periodic
	in each coordinate.
	
	\begin{lemma}[Projection with a gradient constraint]
		\label{lem:projection}
		Let $a_0\in\R^2$, let $t\geq|a_0|$, and let
		$\eta\in W^{1,\infty}_{\per}(Q)$. Set
		$u=a_0\cdot x+\eta$. There is
		$\zeta\in W^{1,\infty}_{\per}(Q)$ such that
		$w=a_0\cdot x+\zeta$ satisfies
		\begin{align}
			|\nabla w|&\leq t\quad\text{a.e.},\label{eq:projected-bound}\\
			\nabla w=&\nabla u\quad\text{or}\quad |\nabla w|=t
			\quad\text{a.e.},\label{eq:projected-alternative}\\
			\int_Q \nabla w & \cdot(\nabla u-\nabla w)\dx \geq0.
			\label{eq:projected-energy}
		\end{align}
	\end{lemma}
	
	We refer to \cite[Lemma~4]{AA} for the one-page proof. In its notation, take
	$K=\{z\in\R^2:|z|\leq t\}$ and let the vector denoted there by $p$
	equal $a_0$. The periodic square is used
	in the proof of \eqref{eq:projected-energy}: in
	\cite[p.~4, proof of (6)]{AA}, the line
	``$\zeta(\,\cdot\pm se_i)\in\mathcal C$'' supplies translated admissible
	competitors, after which the authors change variables on $Q$. This
	translation step is the only place where we use $Q$ as a periodic
	square.

	\subsection{The inequality for \texorpdfstring{$L_t$}{L-t} on the disk}
	\label{sec:cutoff-proof}
	
	The following argument proves Proposition~\ref{prop:cutoff}; its smooth
	periodic part is \cite[Theorem~1]{AA}, rescaled from $L_1$ to $L_t$.
	
	\begin{proof}[Proof of Proposition~\ref{prop:cutoff}]
		\emph{The periodic inequality.}
		For $A\in\R^{2\times2}$ and
		$\psi\in W^{1,2}_{\per}(Q;\R^2)$, periodicity and the determinant
		calculation used to prove \eqref{eq:moments-disk} give
		\begin{equation}\label{eq:moments}
			\int_Q(A+D\psi)\,dx=|Q|A,
			\qquad
			\int_Q\det(A+D\psi)\,dx=|Q|\det A.
		\end{equation}
		First let $f=Ax+\varphi$ with $\varphi$ smooth and periodic on $Q$,
		and write $f=(u,v)$. If $t<M(A)$, \eqref{eq:pointwise-cutoff}, Jensen's
		inequality for the Euclidean norm, and \eqref{eq:moments} give
		\begin{align*}
			\int_Q L_t(Df)\,dx
			&\geq\int_Q\bigl(t|\mathbf a(Df)+\mathbf b(Df)|-t^2\bigr)\,dx\\
			&\geq t\left|\int_Q\bigl(\mathbf a(Df)+\mathbf b(Df)\bigr)\,dx\right|
			-t^2|Q|\\
			&= |Q|\bigl(t(M(A)+m(A))-t^2\bigr)=|Q|L_t(A).
		\end{align*}
		Here $\int_QDf\,dx=|Q|A$ and \eqref{eq:singular-sum-real} give
		\[
		\left|\int_Q\bigl(\mathbf a(Df)+\mathbf b(Df)\bigr)\,dx\right|
		=|Q|\,|\mathbf a(A)+\mathbf b(A)|=|Q|(M(A)+m(A)).
		\]
		Suppose $t\geq M(A)$. Then
		$u=\mathbf a(A)\cdot x+\varphi^1$, and
		$|\mathbf a(A)|\leq M(A)\leq t$. Lemma~\ref{lem:projection} gives
		$w=\mathbf a(A)\cdot x+\zeta$ for which
		$z=\nabla w$ either has $|z|=t$, or equals $\nabla u$ with
		$|\nabla u|\leq t$. By \eqref{eq:selection},
		\[
		L_t(Df)\geq
		\nabla w\cdot(\partial_2v,-\partial_1v)
		+\nabla w\cdot(\nabla u-\nabla w).
		\]
		The first term is $\det D(w,v)$. Since $(w,v)$ is a periodic
		$W^{1,2}$ perturbation of the affine map $Ax$, its integral
		is $|Q|\det A$ by
		\eqref{eq:moments}; the second term has nonnegative integral by
		\eqref{eq:projected-energy}. Thus
		\[
		\int_Q L_t(Df)\,dx\geq |Q|\det A=|Q|L_t(A).
		\]
		
		\medskip
		\emph{Restriction to the disk.}
		We first note that
		\begin{equation}\label{eq:L-Lipschitz}
			|L_t(P)-L_t(\widetilde P)|\leq\sqrt2\,t\,|P-\widetilde P|
			\qquad(P,\widetilde P\in\R^{2\times2}).
		\end{equation}
		Indeed, on the convex set $\{M\leq t\}$, the derivative of the determinant
		has Frobenius norm $|\cof P|=|P|\leq\sqrt2\,t$. The function
		\[
		P\longmapsto t|\mathbf a(P)+\mathbf b(P)|-t^2
		\]
		is globally $\sqrt2\,t$-Lipschitz, since
		$P\mapsto\mathbf a(P)+\mathbf b(P)$ has norm $\sqrt2$. Split a
		matrix segment at the points where it enters and leaves $\{M\leq t\}$;
		the branches agree there, and their Lipschitz estimates give
		\eqref{eq:L-Lipschitz}.
		
		For $\varphi\in C_c^\infty(\D;\R^2)$, extend $\varphi$ by zero to
		$Q$ and cancel the constant exterior contribution $L_t(A)$. This
		gives \eqref{eq:cutoff-integral} for $f=Ax+\varphi$. For a general
		zero-trace perturbation, choose $\varphi_j\in C_c^\infty(\D;\R^2)$
		with $\varphi_j\to f-Ax$ in $W^{1,2}_0(\D;\R^2)$ and use
		\eqref{eq:L-Lipschitz}:
		\[
		\|L_t(A+D\varphi_j)-L_t(Df)\|_{L^1(\D)}
		\leq\sqrt2\,t\,\|D\varphi_j-D(f-Ax)\|_{L^1(\D)}\longrightarrow0.
		\]
	\end{proof}
	
	\section{Distributional admissibility and Hessian measures}
	\label{sec:measures}
	
	The dual-cone definition of $k$-convexity in
	\cite[Lemma~2.2]{TW2} and \cite[Definition~8.1]{WangHessian}
	adapts to our cones. For $1<q\leq2$, call $u\in C(\Omega)$
	distributionally $\Gamma_q$-admissible if
	\begin{equation}\label{eq:distributional-admissibility}
		\int_\Omega u\,\tr(CD^2\varphi)\dx\geq0
		\quad\text{for all constant }C\in\Gamma_q^*
		\text{ and }0\leq\varphi\in C_c^\infty(\Omega).
	\end{equation} 
	For smooth $u$,
	this is equivalent to $D^2u\in\Gamma_q$. The same equivalence holds
	a.e. for $u\in W^{2,q}_{\loc}$, by integration by parts and duality
	of the closed convex cones. Nonnegative convolution preserves
	\eqref{eq:distributional-admissibility}, so it gives smooth
	closed-cone approximations on compact subdomains.
	
	When $q$ is not an integer, admissibility alone does not supply a Hessian measure. For integer
	$k$, Trudinger--Wang construct a measure agreeing with the $k$-Hessian
	on smooth functions and weakly continuous under local convergence in
	measure \cite[Theorem~1.1]{TW2}; see also
	\cite[Theorem~8.1]{WangHessian}. Their argument uses the integral
	structure of the elementary symmetric polynomials. For the present
	noninteger $q$, even locally uniform convergence is insufficient.
	
	\begin{proposition}[Obstruction to a weakly continuous $q$-Hessian measure]
		\label{prop:no-measure}
		Let $1<q<2$. There is no assignment of nonnegative Radon measures
		$u\mapsto\mu_q[u]$ to continuous distributionally
		$\Gamma_q$-admissible functions having both of the following properties:
		\[
		\mu_q[u]=S_q(D^2u)\,dx\quad\text{for smooth }u,
		\qquad
		u_j\to u\text{ locally uniformly }\Longrightarrow
		\mu_q[u_j]\rightharpoonup\mu_q[u].
		\]
	\end{proposition}
	
	\begin{proof}
		Work on any bounded domain $\Omega$ and fix $0<a<1$. Put
		\[
		u(x)=x_1^2+\frac{x_2^2}{2},\qquad
		u_j(x)=u(x)+\frac{a}{2^{2j}}\bigl(2-\cos(2^jx_1)\bigr).
		\]
		Then $u_j\to u$ in $C^1$, and
		\[
		D^2u_j=\operatorname{diag}(2+a\cos(2^jx_1),1)>0;
		\]
		in particular $D^2 u_j \in \Gamma_q$.
		For $t>1$, let
		\[
		g(t)=S_q(\operatorname{diag}(t,1))
		=\frac q2t^{q-1}+\left(1-\frac q2\right)t^q.
		\]
		We have
		\[
		S_q(D^2 u_j) \, dx = g(2+a \cos(2^j x_1)) \, dx \rightharpoonup\bar g\,dx,
		\qquad
		\bar g=\frac1{2\pi}\int_0^{2\pi}g(2+a\cos\theta)\,d\theta;
		\]
		in fact, if $\phi \in C_c(\Omega)$, then we split the $x_1$ axis into intervals of length $2\pi/2^j$, and on each such interval $I = [c,c+2\pi/2^j]$, 
		\[
		\begin{split}
			&\quad \left| \int_{\substack{x \in \R^2\\ x_1 \in I}}  [g(2+a \cos(2^j x_1))-\bar{g}] \phi(x) dx \right| \\
			&= \left| \int_{x_1 \in I}  [g(2+a \cos(2^j x_1))-\bar{g}] \left[\int_{\R} \phi(x_1,x_2) dx_2 -\int_{\R}\phi(c,x_2) dx_2 \right] dx_1 \right|\\
			&\leq 3^q C |I| \omega_{\phi}(2\pi/2^j) 
		\end{split}
		\]
		where $C$ is the diameter of the projection of $\Omega$ on the $x_2$ axis, and $$\omega_{\phi}(\delta) := \sup_{\substack{x,x' \in \R^2 \\ |x-x'| \leq \delta}} |\phi(x)-\phi(x')| $$ is the modulus of continuity of $\phi$. Summing over the finitely many $I$ that intersects the projection of $\Omega$ on the $x_1$ axis, and using that $\omega_{\phi}(\delta) \to 0$ as $\delta \to 0^+$, we obtain the desired weak convergence of $S_q(D^2 u_j)\, dx$ to $\bar{g}\, dx$. 
		
		Now note
		\[
		g''(t)=\frac{q(q-1)(2-q)}2\,t^{q-3}(t-1)>0.
		\]
		As a result, Jensen's inequality shows
		\[
		\bar g=\frac1{2\pi}\int_0^{2\pi}g(2+a\cos\theta)\,d\theta
		>g(2)=S_q(D^2u).
		\]
		This shows $\bar{g} \, dx \ne S_q(D^2 u)\, dx$; in particular, $S_q(D^2 u_j) \, dx$ does not converge weakly to $S_q(D^2 u) \, dx$.
	\end{proof}

\end{document}